[1994/12/01]
\documentclass{ijmart-mod}
\chardef\bslash=`\\ 

\usepackage{graphicx}
\usepackage[breaklinks=true]{hyperref}
\usepackage{mathtools}
\usepackage{xcolor}

\newtheorem{thm}{Theorem}[section]
\newtheorem{cor}[thm]{Corollary}
\newtheorem{lem}[thm]{Lemma}
\newtheorem{prop}[thm]{Proposition}

\theoremstyle{definition}

\newtheorem{rem}[thm]{Remark}
\newtheorem{qtn}[thm]{Question}

\theoremstyle{remark}

\newcommand{\eval}[2][\right]{\relax
  \ifx#1\right\relax \left.\fi#2#1\rvert}

\begin{document}

\title{The expansion of half-integral polytopes}

\author[J. Cardinal]{Jean Cardinal}
\address{Universit{\'e} libre de Bruxelles (ULB), Belgium}
\email{jean.cardinal@ulb.be}

\author[L. Pournin]{Lionel Pournin}
\address{Universit{\'e} Paris 13, Villetaneuse, France}
\email{lionel.pournin@univ-paris13.fr}


\begin{abstract}
The expansion of a polytope is an important parameter for the analysis of the random walks on its graph. A conjecture of Mihai and Vazirani states that all $0/1$-polytopes have expansion at least 1. We show that the generalization to half-integral polytopes does not hold by constructing $d$-dimensional half-integral polytopes whose expansion decreases exponentially fast with $d$. We also prove that the expansion of half-integral zonotopes is uniformly bounded away from~$0$. As an intermediate result, we show that half-integral zonotopes are always graphical.
\end{abstract}

\maketitle

\section{Introduction}
\label{sec:intro}


Consider a graph $G$ with vertex set $V$. The \emph{Cheeger constant} of $G$, also known as the \emph{edge expansion} of $G$ is the quantity 
\begin{equation}\label{sec:intro.eq.1}
h(G)=\min \biggl\{ \frac{|\partial{S}|}{\min\{|S|,|V\mathord{\setminus{S}}|\}} : S\subset V\biggl\} ,
\end{equation}
where $\partial{S}$ denotes the set of the edges of $G$ with a vertex in $S$ and the other in $V\mathord{\setminus}S$. This quantity allows to determine whether $G$ has a bottleneck in the sense that only a small set of its edges separates two large subsets of its vertices. Similar isoperimetric numbers can be defined. For example, replacing $\partial{S}$ in (\ref{sec:intro.eq.1}) with the number of vertices of $S$ incident to an edge of $G$ whose other vertex is not in $S$ gives rise to the vertex expansion of $G$. This kind of isopetrimetric numbers have many applications. In particular, hey play an important role in the analysis of random walks and of algorithms for sampling and counting combinatorial objects (see \cite{HooryLinialWigderson2006} for a comprehensive overview).

In this article, we are mainly interested in the notion of edge expansion and will refer to this notion for short as the \emph{expansion} of a graph. The graphs we investigate are \emph{polytopal} graphs. More precisely, if $P$ is a polytope---a convex hull of a finite subset of $\mathbb{R}^d$---, then its vertices and edges form a graph which we call the \emph{graph of $P$}. We will indifferently refer to the expansion of that graph as the expansion of $P$. A famous conjecture of Mihail and Vazirani (see~\cite{FederMihail1992,Kaibel2004} and \cite[Conjecture 19.2.11]{Kalai2017}) states that $0/1$-polytopes---convex hulls of points from $\{0,1\}^d$---always have expansion at least 1. The conjecture has been thoroughly studied from different points of view \cite{Kaibel2004,AnariLiuGharanVinzant2019,KwokLauTung2022,LerouxRademacher2023} and is known to hold for sereval families of $0/1$-polytopes as for instance simple $0/1$-polytopes~\cite{Kaibel2004} or base polytopes of matroids~\cite{AnariLiuGharanVinzant2019} but remains open in general.

We show that if instead of asking that the vertices of a polytope $P$ belong to $\{0,1\}^d$, we allow them to belong to $\{0,1/2,1\}^d$, then we can construct polytopes whose edge expansion goes exponentially fast to $0$ as $d$ grows large. Polytopes whose vertices are drawn from $\{0,1/2,1\}^d$ are called \emph{half-integral}, and are of wide interest in combinatorics and optimization~\cite{BraunPokutta2011,DelPiaMichini2016,GuoJerrum2023}.

Our first main result is the following.

\begin{thm}
\label{thm:hip}
There exists half-integral polytopes of arbitrarily large dimension $d$ whose expansion is less than $d/\sqrt{2}^d$.
\end{thm}

The subclass of half-integral \emph{zonotopes}---half-integral polytopes whose all faces are centrally-symmetric---also plays an important role in combinatorics~\cite{DelPiaMichini2016,DezaManoussakisOnn2018,DezaPournin2022}. Our second main result is that the expansion of half-integral zonotopes is bounded away from $0$ independently on the dimension. 

\begin{thm}\label{thm:hiz}
Half-integral zonotopes have expansion at least $7/12$.
\end{thm}

A graphical zonotope is the Minkowski sum of a subset of the edges of the standard simplex~\cite{GreeneZaslavsky1983,PadrolPilaudPoullot2023}. Each graphical zonotope is thus naturally associated with the graph $G$ induced by these edges in the graph of the standard simplex and one refers to it as the graphical zonotope of $G$.

As an intermediate step to proving Theorem~\ref{thm:hiz}, we show that all half-integral zonotopes are graphical in the following sense.

\begin{thm}\label{thm:grz}
Every half-integral zonotope coincides, up to an affine, bijective transformation with a graphical zonotope.
\end{thm}

Theorem \ref{thm:grz} may be of independent interest.

Half-integral polytopes and $0/1$-polytopes can be generalized as follows. For any fixed positive integer $k$, a $1/k$-integral polytope is a polytope whose coordinates of vertices are of the form $i/k$ where $i$ is an integer satisying $0\leq{i}\leq{k}$. In view of our results it is natural to the following question.

\begin{qtn}\label{qtn:hiz}
Is the expansion of $1/k$-integral zonotopes bounded away from $0$ by a function that only depends on $k$ (and not on the dimension)?
\end{qtn}

As we shall see, when $k$ is greater than $2$, not all $1/k$-integral zonotopes are graphical, up to an affine transformation. Hence Question \ref{qtn:hiz} can be specialized, in a strict way, to graphical zonotopes. This is particularly relevant because $1/k$-integral graphical zonotopes are precisely the graphical zonotopes of graphs whose degree of all vertices is at most $k$.

\begin{qtn}\label{qtn:grz}
Is the expansion of the graphical zonotopes of graphs whose degree of all vertices is at most $k$ bounded away from $0$ by a function that only depends on $k$ (and not otherwise on the associated graphs)?
\end{qtn}

The article is organized as follows. Section~\ref{CP.sec.2} is dedicated to constructing half-integral polytopes with exponentially small expansion and to the proof of Theorem~\ref{thm:hip}. In Section~\ref{sec:gz}, we prove half-integral zonotopes are graphical in the sense of Theorem \ref{thm:grz}. In Section~\ref{sec:pcp}, we bound the expansion of the graphical zonotopes of cycles using the method described in \cite{Sinclair92} (see also \cite{Kaibel2004}) and establish Theorem~\ref{thm:hiz} as a consequence.

\section{Sparse cuts in half-integral polytopes}\label{CP.sec.2}

In this section, we assume that $d$ is an odd integer such that $(d-1)/2$ is also odd. The main goal of this section is to construct a family of $d$-dimensional half-integral polytopes whose expansion gets exponentially small as $d$ goes to infinity. Consider the hyperplane $H_k$ made of all the points in $\mathbb{R}^d$ whose coordinates sum to $k$. The intersection of the hypercube $[0,1]^d$ with $H_k$ is the hypersimplex $\Delta(d,k)$. Equivalently $\Delta(d,k)$ is the polytopes whose vertices are the points in $\{0,1\}^d$ with exactly $k$ non-zero coordinates.

We are interested in the slab $L$ of $\mathbb{R}^d$ bounded by $H_{(d-1)/2}$ and $H_{(d+1)/2}$:
$$
L=\left\{x\in\mathbb{R}^d:\frac{d-1}{2}\leq\sum_{i=0}^dx_i\leq\frac{d+1}{2}\right\}\mbox{.}
$$

Let $\mathcal{C}$ be the set of all the centers of the $(d-1)/2$-dimensional faces of the $d$-dimensional unit hypercube $[0,1]^d$. We consider the set
$$
\mathcal{V}=\bigl(\mathcal{C}\mathord{\setminus}L\bigr)\cup\bigl(\{0,1\}^d\cap{L}\bigr)\mbox{.}
$$

Informally, $\mathcal{V}$ is made of the vertices of the hypercube $[0,1]^d$ whose sum of coordinates is $(d-1)/2$ or $(d+1)/2$ and of the centers of the $(d-1)/2$-dimensional faces of the hypercube whose sum of coordinates is less than $(d-1)/2$ or greater than $(d+1)/2$. The announced polytope is the convex hull of $\mathcal{V}$ and it will be denoted by $\Xi$. This polytope is shown in Figure \ref{CP.sec.2.fig.1} when $d$ is equal to $3$. In that case, $\Xi$ is just a cube that has been truncated at two opposite vertices. The vertices of the two triangular faces of $\Xi$ are centers of edges of the cube and the other vertices of $\Xi$ are vertices of the cube.

We first examine the vertices of $\Xi$.

\begin{figure}[t]
\begin{centering}
\includegraphics[scale=1]{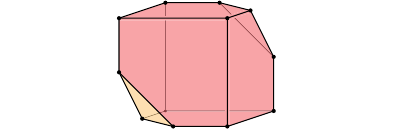}
\caption{The polytope $\Xi$ when $d$ is equal to $3$.}\label{CP.sec.2.fig.1}
\end{centering}
\end{figure}

\begin{lem}\label{CP.sec.2.lem.1}
The vertex set of $\Xi$ is precisely $\mathcal{V}$.
\end{lem}
\begin{proof}
Since $\Xi$ is the convex hull of $\mathcal{V}$, it suffices to show that each point in $\mathcal{V}$ is a vertex of $\Xi$. First pick a point $v$ in $\{0,1\}^d\cap{L}$. As $v$ is a vertex of the hypercube $[0,1]^d$, the singleton $\{v\}$ is the intersection of $[0,1]^d$ with some hyperplane $H$ of $\mathbb{R}^d$. Since $\Xi$ is, by definition, a subset of $[0,1]^d$, that singleton is also the intersection of $\Xi$ with $H$ and, therefore, $v$ is a vertex of $\Xi$.

Now pick a point $v$ in $\mathcal{C}\mathord{\setminus}L$ and consider the $(d-1)/2$-dimensional face $F$ of $[0,1]^d$ whose center is $v$. Observe that $\Xi\cap{F}$ is non-empty because this intersection contains $v$. As in addition, $\Xi$ is a subset of $[0,1]^d$, it must admit $\Xi\cap{F}$ as a proper face. Hence, it suffices to show that $v$ is a vertex of $\Xi\cap{F}$. Observe that the only possible vertices of $\Xi$ contained in $\Xi\cap{F}$ are $v$ and the vertices of $L\cap{F}$. Denote by $s$ the sum of the coordinates of $v$ and consider the hyperplane $H$ of $\mathbb{R}^d$ made of the points whose sum of coordinates is $s$. Since $v$ does not belong to $L$, $s$ is either less than $(d-1)/2$ or greater than $(d+1)/2$. Hence, $H$ is disjoint from $L$ and the intersection of $\Xi\cap{F}$ and $H$ is $\{v\}$. This shows that $v$ is a vertex of $\Xi\cap{F}$ and, therefore, a vertex of $\Xi$.
\end{proof}

Let us now count the vertices of $\Xi$. First observe that $\{0,1\}^d\cap{L}$ is the union of the vertex sets of the hypersimplices $\Delta(d,(d-1)/2)$ and $\Delta(d,(d+1)/2)$. Assuming that $\mathbb{R}^d$ is the hyperplane of $\mathbb{R}^{d+1}$ spanned by the first $d$ coordinates, this set is also the orthogonal projection of the vertex set of the hypersimplex $\Delta(d+1,(d+1)/2)$ on $\mathbb{R}^d$. As an immediate consequence,
\begin{equation}\label{CP.sec.2.eq.1}
|\{0,1\}^d\cap{L}|={d+1\choose{(d+1)/2}}\mbox{.}
\end{equation}

There remains to count the points in $\mathcal{C}\mathord{\setminus}L$. 

\begin{lem}\label{CP.sec.2.lem.2}
The number of points in $\mathcal{C}\mathord{\setminus}L$ is
$$
{d\choose{(d-1)/2}}\biggl[2^{(d+1)/2}-{{(d+1)/2}\choose{(d+1)/4}}\biggr]\mbox{.}
$$
\end{lem}
\begin{proof}
Recall that the $d$-dimensional hypercube has
$$
2^{d-k}{d\choose{k}}
$$
$k$-dimensional faces. The number of points in $\mathcal{C}$ is precisely the number of the $(d-1)/2$-dimensional faces of the $d$-dimensional hypercube. Hence,
\begin{equation}\label{CP.sec.2.lem.2.eq.1}
|\mathcal{C}|=2^{(d+1)/2}{d\choose{(d-1)/2}}\mbox{.}
\end{equation}

Now consider a point $v$ in $\mathcal{C}\cap{L}$. This point has $(d-1)/2$ coordinates equal to $1/2$ and its other coordinates are either equal to $0$ or to $1$.

However, as $v$ belongs to $L$,
$$
\frac{d-1}{2}\leq\sum_{i=1}^dv_i\leq\frac{d+1}{2}\mbox{.}
$$

Hence, denoting by $k$ the number of coordinates of $v$ that are equal to $1$,
$$
\frac{d-1}{2}\leq\frac{d-1}{4}+k\leq\frac{d+1}{2}
$$
and, in turn,
\begin{equation}\label{CP.sec.2.lem.2.eq.2}
\frac{d+1}{4}-\frac{1}{2}\leq{k}\leq\frac{d+1}{4}+\frac{1}{2}\mbox{.}
\end{equation}

By our assumption that $(d-1)/2$ is odd, $(d+1)/4$ is an integer and (\ref{CP.sec.2.lem.2.eq.2}) implies that $k$ is precisely equal to that integer. This shows that $\mathcal{C}\cap{L}$ is the number of points in $\mathbb{R}^d$ with $(d-1)/2$ coordinates equal to $1/2$, $(d+1)/4$ coordinates equal to $1$, and all other coordinates equal to $0$. Hence,
$$
|\mathcal{C}\cap{L}|={d\choose{(d-1)/2}}{{(d+1)/2}\choose{(d+1)/4}}\mbox{.}
$$

Subtracting the right-hand side of this equality from the right-hand side of (\ref{CP.sec.2.lem.2.eq.1}) results in the desired expression for $|\mathcal{C}\mathord{\setminus}L|$. 
\end{proof}

We now turn our attention to the edges of $\Xi$. We will not need to determine all of them. Indeed, we will only be interested in counting the number of edges that cross the hyperplane $H_{d/2}$ through the center of $[0,1]^d$.

\begin{lem}\label{CP.sec.2.lem.4}
The edges of $\Xi$ whose relative interior is non-disjoint from the hyperplane $H_{d/2}$ are precisely the edges of $[0,1]^d$ that share a vertex with $\Delta(d,(d-1)/2)$ and the other with $\Delta(d,(d+1)/2)$.
\end{lem}
\begin{proof}
Consider an edge $e$ of $\Xi$ and assume that the relative interior of $e$ is non-disjoint from $H_{d/2}$. Since the coordinates of a vertex of $\Xi$ sum to either at most $(d-1)/2$ or to at least $(d+1)/2$ and the edge $e$ contains a point whose coordinates sum to $d/2$, this edge has a vertex $u$ whose sum of coordinates is at most $(d-1)/2$ and a vertex $v$ whose sum of coordinates is at least $(d+1)/2$. We first prove that $u$ belongs to $H_{(d-1)/2}$ and $v$ to $H_{(d+1)/2}$. By symmetry, it suffices to show that $u$ belongs to $H_{(d-1)/2}$.

Assume for contradiction that $u$ does not belong to $H_{(d-1)/2}$ or in other words that the sum of its coordinates is less than $(d-1)/2$. In that case, a point $x$ in the relative interior of $e$ belongs to $H_{(d-1)/2}$. By convexity, this point also belongs to the hypercube $[0,1]^d$ and we get that $x$ is in the intersection $[0,1]^d\cap{H_{(d-1)/2}}$. However, this intersection is precisely $\Delta(d,(d-1)/2)$. As a consequence, $x$ is a convex combination of the vertices of that hypersimplex, which contradicts our choice that $e$ is an edge of $\Xi$.

It follows from the above that $e$ is an edge of $\Xi$ contained in $L$.  By construction, $\Xi\cap{L}$ coincides with $[0,1]^d\cap{L}$ and therefore, $e$ must be an edge of the hypercube $[0,1]^d$. Inversely, all the edges of $[0,1]^d$ that share a vertex with $\Delta(d,(d-1)/2)$ and the other with $\Delta(d,(d+1)/2)$ are contained in $L$. Hence any such edge of $[0,1]^d$ must be an edge of $\Xi$.
\end{proof}

\begin{thm}\label{CP.sec.2.thm.1}
For any sufficiently large integer $d$ such that both $d$ and $(d-1)/2$ are odd, the expansion of $\Xi$ is less than $d/\sqrt{2}^d$.
\end{thm}
\begin{proof}
Assume that $d$ is odd and denote by $S$ the subset of the vertices of $\Xi$ whose coordinates sum to at most $(d-1)/2$. By symmetry, $S$ contains exactly half of the vertices of $\Xi$ and, by Lemma \ref{CP.sec.2.lem.2},
\begin{equation}\label{CP.sec.2.thm.1.eq.1}
|S|=\frac{1}{2}{d\choose{(d-1)/2}}\biggl[2^{(d+1)/2}-{{(d+1)/2}\choose{(d+1)/4}}\biggr]\mbox{.}
\end{equation}

It is well-known that the central binomial coefficient can be bounded as
\begin{equation}\label{CP.sec.2.thm.1.eq.2}
{(d+1)/2\choose(d+1)/4}\leq\frac{2\sqrt{2}^{d+1}}{\sqrt{\pi}(d+1)}\mbox{.}
\end{equation}

One obtains, by combining (\ref{CP.sec.2.thm.1.eq.1}) and (\ref{CP.sec.2.thm.1.eq.2}) that 
$$
|S|\geq\frac{1}{2}{d\choose{(d-1)/2}}\sqrt{2}^{d+1}\biggl(1-\frac{2}{\sqrt{\pi}(d+1)}\biggr)\mbox{.}
$$

Denote by $\partial{S}$ the subset of the edges of $\Xi$ with a vertex in $S$ and the other not. It follows from Lemma \ref{CP.sec.2.lem.4} that
$$
|\partial{S}|=\frac{d+1}{2}{d\choose{(d-1)/2}}\mbox{.}
$$

As a consequence,
\begin{equation}\label{CP.sec.2.thm.1.eq.4}
\frac{|\partial{S}|}{|S|}\leq\frac{d+1}{\displaystyle\sqrt{2}^{d+1}\biggl(1-\frac{2}{\sqrt{\pi}(d+1)}\biggr)}\mbox{.}
\end{equation}

Since, for any sufficiently large $d$,
$$
\frac{d+1}{\displaystyle\sqrt{2}\biggl(1-\frac{2}{\sqrt{\pi}(d+1)}\biggr)}
$$
is less than $d$, the result follows from (\ref{CP.sec.2.thm.1.eq.4}).
\end{proof}

Note that Theorem \ref{thm:hip} is an immediate consequence of Theorem \ref{CP.sec.2.thm.1}.

\section{Half-integral zonotopes are graphical}
\label{sec:gz}

Recall that a zonotope is a Minkowski sum of finitely-many segments from $\mathbb{R}^d$. In other words, a zonotope $Z$ is, up to translation the Minkowski sum
$$
\sum_{g\in\mathcal{G}}\mathrm{conv}\{0,g\}
$$
where $\mathcal{G}$ is a finite set of vectors from $\mathbb{R}^d$. It will be assumed without loss of generality that the vectors in $\mathcal{G}$ are pairwise non-collinear and that their first non-zero coordinate is positive. By these assumptions, $\mathcal{G}$ is unique and the vectors in $\mathcal{G}$ will be referred to as the \emph{generators} of $Z$.

All zonotopes are polytopes and, just as polytopes, they are half-integral when the coordinates of their vertices are $0$, $1/2$, or $1$. In this section, we relate half-integral zonotopes with the sub-class of \emph{graphical zonotopes}. Let us first recall that graphical zonotopes are precisely the Minkowski sums of subsets of edges of a standard simplex. Each graphical zonotope corresponds to a graph in the following sense. A graph $G$ with $d$ vertices can be embedded in the $1$-skeleton of the $(d-1)$-dimensional standard simplex $\Delta_{d-1}$ by mapping the vertices of $G$ to the vertices of $\Delta_{d-1}$. The graphical zonotope $Z(G)$ of $G$ is then the Minkowski sum of the edges of $\Delta_{d-1}$ that correspond to the edges of $G$ via this embedding. This zonotope depends on how the vertices of $G$ are mapped to those of $\Delta_{d-1}$ but, by the symmetries of $\Delta_{d-1}$, all the zonotopes resulting from different such mappings coincide up to isometry and we will call any of them the graphical zonotope of $G$. We will need the following statement.

\begin{prop}\label{sec:gz.prop.0}
Consider a half-integral zonotope $Z$ contained in $\mathbb{R}^d$. For every integer $i$ satisfying $1\leq{i}\leq{d}$, there are at most two generators of $Z$ whose $i$th coordinate is non-zero. Moreover, if there are exactly two such generators of $Z$, then the absolute value of their $i$th coordinate is $1/2$.
\end{prop}
\begin{proof}
Consider the $1$-dimensional linear subspace $L_i$ of $\mathbb{R}^d$ spanned by the $i$th coordinate. The orthogonal projection of $Z$ on $L_i$ is a line segment of length
$$
\sum_{g\in\mathcal{G}}|g_i|\mbox{.}
$$

As $Z$ is contained in $[0,1]^d$, this length is at most $1$. Since the absolute values of the non-zeros coordinates of the generators of $Z$ are $1/2$ or $1$, it follows that at most two generators of $Z$ can have a non-zero $i$th coordinate and when there are two such generators, the absolute value of their $i$th coordinate is $1/2$.
\end{proof}

Recall that a set of vectors from $\mathbb{R}^d$ is \emph{minimally} linearly dependent when it is linearly dependent but all its proper subsets are linearly independent. Let us state properties of half-integral zonotopes that follow from Proposition \ref{sec:gz.prop.0}.

\begin{lem}\label{sec:gz.lem.0}
Consider a half-integral zonotope $Z$. If $\mathcal{A}$ is a minimally linearly dependent subset of the generators of $Z$ then there exists a family $(\lambda_g)_{g\in\mathcal{A}}$ of numbers, each of them equal to $1$ or to $-1$ satisfying
\begin{equation}\label{sec:gz.lem.0.eq.0}
\sum_{g\in\mathcal{A}}\lambda_gg=0\mbox{.}
\end{equation}
\end{lem}
\begin{proof}
If $\mathcal{A}$ is a minimally linearly dependent subset of the generators of $Z$ then there exists a family $(\lambda_g)_{g\in\mathcal{A}}$ of numbers such that (\ref{sec:gz.lem.0.eq.0}) holds. It suffices to show that the absolute value of these numbers are all equal. Let $\mathcal{B}$ be a non-empty subset of $\mathcal{A}$ such that any two distinct vectors of $\mathcal{B}$ have identical or opposite coefficients in the left-hand side of (\ref{sec:gz.lem.0.eq.0}). Assume that $\mathcal{B}$ is maximal for inclusion with respect to this requirement. We will show that $\mathcal{B}$ must then coincide with $\mathcal{A}$. Let us assume, for contradiction that $\mathcal{B}$ is not equal to $\mathcal{A}$.

Observe first that the linear spans of $\mathcal{B}$ and of $\mathcal{A}\mathord{\setminus}\mathcal{B}$ cannot be orthogonal. Indeed, since $\mathcal{A}$ is minimally affinely dependent, $\mathcal{A}\mathord{\setminus}\mathcal{B}$ is affinely independent. Moreover, any vector $x$ in $\mathcal{B}$ can be written as a linear combination of $\mathcal{A}\mathord{\setminus}\{x\}$. Hence, if the linear spans of $\mathcal{B}$ and $\mathcal{A}\mathord{\setminus}\mathcal{B}$ were orthogonal, the coefficients of the vectors from $\mathcal{A}\mathord{\setminus}\mathcal{B}$ in that linear combination would be equal to $0$ and $x$ could be expressed as a linear combination of $\mathcal{B}\mathord{\setminus}\{x\}$. This would imply that $\mathcal{B}$ is affinely dependent and contradict the minimal affine dependence of $\mathcal{A}$.

Since the linear spans of $\mathcal{B}$ and of $\mathcal{A}\mathord{\setminus}\mathcal{B}$ are not orthogonal, there exists a vector $x$ in $\mathcal{A}\mathord{\setminus}\mathcal{B}$ that is not orthogonal to every vector from $\mathcal{B}$. Hence, assuming that $Z$ is contained in $\mathbb{R}^d$, there exists an integer $i$ satisfying $1\leq{i}\leq{d}$ and a vector $y$ in $\mathcal{B}$ such that $x_i$ and $y_i$ are both non-zero. By Proposition~\ref{sec:gz.prop.0}, the $i$th coordinate of the other vectors in $\mathcal{S}$ is equal to $0$ and (\ref{sec:gz.lem.0.eq.0}) implies
\begin{equation}\label{sec:gz.lem.0.eq.1}
\lambda_xx_i+\lambda_yy_i=0\mbox{.}
\end{equation}

However, by Proposition \ref{sec:gz.prop.0} again, the absolute values of $x_i$ and $y_i$ are both equal to $1/2$. Hence, it follows from (\ref{sec:gz.lem.0.eq.1}) that the absolute values of $\lambda_x$ and $\lambda_y$ coincide, which contradicts our assumption on $\mathcal{B}$.
\end{proof}

The following is another consequence of Proposition \ref{sec:gz.prop.0}.

\begin{lem}\label{sec:gz.lem.1}
Consider a half-integral zonotope $Z$ and denote by $\mathcal{G}$ its set of generators. If $\mathcal{A}$ is a minimally linearly dependent subset of $\mathcal{G}$, then the linear spans of $\mathcal{A}$ and of $\mathcal{G}\mathord{\setminus}\mathcal{A}$ are orthogonal.
\end{lem}
\begin{proof}
Assume that $\mathcal{A}$ is a minimally linearly dependent subset of $\mathcal{G}$. Consider an integer $i$ satisfying $1\leq{i}\leq{d}$, a vector $x$ in $\mathcal{A}$ and a vector $y$ in $\mathcal{G}\mathord{\setminus}\mathcal{A}$. We will show that if $x_i$ is non-zero, then $y_i$ necessarily vanishes. Observe that the lemma immediately follows from this property. 

According to Lemma \ref{sec:gz.lem.0},
$$
\sum_{g\in\mathcal{A}}\lambda_gg=0
$$
where $(\lambda_g)_{g\in\mathcal{A}}$ is a family of non-zero coefficients. If $x_i$ is non-zero, it follows that there must be a vector $z$ in $\mathcal{A}$ distinct from $x$ such that $z_i$ is non-zero. By Proposition \ref{sec:gz.prop.0}, $x$ and $z$ must then be the only generators of $Z$ whose $i$th coordinate is non-zero. Hence, $y_i$ must be equal to $0$, as desired.
\end{proof}

In the sequel, by an \emph{affine transformation}, we refer to an affine map between two affine spaces that is also a bijection.

\begin{lem}\label{sec:gz.lem.2}
Consider a half-integral zonotope $Z$ and denote by $\mathcal{G}$ its set of generators. If $\mathcal{G}$ is minimally linearly dependent, then $Z$ is equal, up to an affine transformation, to the graphical zonotope of a cycle on $|\mathcal{G}|$ vertices.
\end{lem}
\begin{proof}
Denote by $e^1$ to $e^{|\mathcal{G}|}$ the vectors of the canonical basis of $\mathbb{R}^{|\mathcal{G}|}$. Consider the set $\mathcal{G}'$ whose elements are the vector $e^1-e^{|\mathcal{G}|}$ and the vectors $e^i-e^{i+1}$ when $1\leq{i}<|\mathcal{G}|$. Observe that $\mathcal{G}'$ is the set of the generators of the graphical zonotope of a cycle on $|\mathcal{G}|$ vertices, which we will denote by $Z'$.

Now assume that $\mathcal{G}$ is minimally linearly dependent and denote by $x$ a vector in $\mathcal{G}$. According to Lemma \ref{sec:gz.lem.0} there exists a family $(\lambda_g)_{g\in\mathcal{G}\mathord{\setminus}\{x\}}$ of numbers, each of whose is equal to $1$ or to $-1$ that satisfy
\begin{equation}\label{sec:gz.lem.2.eq.1}
\sum_{\substack{g\in\mathcal{G}\\g\neq{x}}}\lambda_gg=x\mbox{.}
\end{equation}

Let us consider the set
$$
\mathcal{A}=\bigl\{\lambda_gg:g\in\mathcal{G}\mathord{\setminus}\{x\}\bigr\}\cup\bigl\{x\bigr\}\mbox{.}
$$

Observe that $Z$ coincides, up to an affine transformation with
$$
\sum_{g\in\mathcal{A}}\mathrm{conv}\{0,g\}\mbox{.}
$$

As a consequence, it suffices to show that there exists a linear bijection from the linear span of $\mathcal{A}$ to that of $\mathcal{G}'$ that sends $\mathcal{A}$ to $\mathcal{G}'$.

Observe that (\ref{sec:gz.lem.2.eq.1}) can be rewritten into
\begin{equation}\label{sec:gz.lem.2.eq.2}
\sum_{\substack{g\in\mathcal{A}\\g\neq{x}}}g=x\mbox{.}
\end{equation}

Since $\mathcal{G}$ is minimally linearly dependent, $\mathcal{G}\mathord{\setminus}\{x\}$ is linearly independent. Since $\mathcal{A}$ is obtained from $\mathcal{G}$ by negating some vectors and keeping the other vectors, $\mathcal{A}\mathord{\setminus}\{x\}$ is linearly independent as well. Hence, by (\ref{sec:gz.lem.2.eq.2}), $\mathcal{A}\mathord{\setminus}\{x\}$ forms a basis of the linear span of $\mathcal{A}$. Likewise, observe that the vectors $e^1-e^2$ to $e^{|\mathcal{G}|-1}-e^{|\mathcal{G}|}$ collectively form a basis of the linear span of $\mathcal{G}'$. Consider a linear bijection $\psi$ from the linear span of $\mathcal{A}$ to that of $\mathcal{G}'$ that sends $\mathcal{A}$ to $\mathcal{G}'$. By linearity, it follows from (\ref{sec:gz.lem.2.eq.2}) that $\psi$ sends $x$ to $e^1-e^{|\mathcal{G}|}$. Hence, $\psi$ is a linear bijection from the linear span of $\mathcal{A}$ to that of $\mathcal{G}'$ that sends $\mathcal{A}$ to $\mathcal{G}'$, as desired. 
\end{proof}

\begin{rem}\label{sec:gz.rem.1}
By the argument in the proof of Lemma~\ref{sec:gz.lem.2}, the graphical zonotope of a cycle on $d$ vertices is, up to an affine transformation, the Minkowski sum of a $(d-1)$-dimensional hypercube with any of its diagonals.
\end{rem}

We are now ready to prove the main result of the section. In order to do that, we will use the straightforward property that the cartesian product of the graphical zonotopes of two graphs is, up to an affine transformation, the graphical zonotope of the disjoint union of these graphs.
\begin{thm}\label{thm:hizg}
Half-integral zonotopes are, up to an affine transformation, graphical zonotopes of graphs of degree at most two.
\end{thm}
\begin{proof}
Consider a half-integral zonotope $Z$ and denote by $\mathcal{G}$ its set of generators. The proof is by induction on $|\mathcal{G}|$. If $\mathcal{G}$ is linearly independent, then it $Z$ is equal, up to an affine transformation to a hypercube. As a hypercube is equal, up to an affine transformation to the graphical zonotope of a path, the result is immediate in this case. This provides the base case of our induction since $\mathcal{G}$ is linearly independent when it is a singleton.

Now assume that $\mathcal{G}$ is linearly dependent (and therefore, non-empty). In that case, $\mathcal{G}$ admits a minimally linearly dependent subset $\mathcal{A}$. If $\mathcal{A}$ is equal to $\mathcal{G}$, then the result is obtained from Lemma \ref{sec:gz.lem.2} and we therefore assume that $\mathcal{A}$ is a proper subset of $\mathcal{G}$. In that case, $Z$ is the Minkowski sum of two half-integral zonotopes, whose sets of generators are $\mathcal{A}$ and $\mathcal{G}\mathord{\setminus}\mathcal{A}$. However, by Lemma \ref{sec:gz.lem.1}, the linear spans of these two zonotopes are orthogonal and as a consequence, $Z$ is equal, up to an affine transformation to the cartesian product of these zonotopes. By induction, these two zonotopes coincide, up to an affine transformation, with the graphical zonotopes of graphs of degree at most two. Hence, the theorem follows from the property that the cartesian product of the graphical zonotopes of two graphs is, up to an affine transformation, the graphical zonotope of the disjoint union of these graphs.
\end{proof}

Note that Theorem \ref{thm:grz} immediately follows from Theorem \ref{thm:hizg}.

\begin{rem}
It is noteworthy that Theorem \ref{thm:hizg} does not carry over to the case of $1/k$-integral zonotopes when $k$ is greater than $2$. Indeed there is a $1/3$-integral octagonal zonogon. However a graphical zonotope cannot have octagonal faces because it is a Minkowski sum of edges of a simplex.
\end{rem}

Recall that the graphs of degree at most two are precisely the disjoint unions of paths and cycles. Hence, Theorem~\ref{thm:hizg} can be rephrased as follows.
\begin{cor}\label{cor:hiz}
Half-integral zonotopes are, up to an affine transformation, cartesian products of hypercubes and graphical zonotopes of cycles.
\end{cor}

\section{The expansion of half-integral zonotopes}\label{sec:pcp}

In this section, we derive a lower bound on the expansion of half-integral zonotopes as a consequence of Theorem \ref{thm:hizg}. In order to do so, we will use the method described in~\cite{Sinclair92} (see also \cite[Section 4.1]{Kaibel2004}). Let us first briefly recall how this method allows to bound the expansion of a graph $G$ with $n$ vertices.

Consider the oriented graph $H$ with the same vertices than $G$ obtained by replacing each edge $e$ of $G$ by two arcs with the same vertices than $e$ and opposite orientations. For any ordered pair $(s,t)$ of vertices of $H$, we consider a unit flow from $s$ to $t$ along oriented paths in $H$. Several such flows may pass through an arc of $H$ and for each such arc $a$, we consider the sum $\phi(a)$ of the flows passing through $a$. Denote by $\phi_{\max}$ the maximum value of $\phi(a)$ when $a$ ranges over the arcs of $H$. By construction, the total quantity flowing out of a subset $S$ of vertices of $H$ is at least $|S|(n-|S|)$ and cannot be greater than $|\partial S|\phi_{\max}$. Hence, under the assumption that $|S|$ is at most $n/2$,
$$
\frac{|\partial S|}{|S|}\geq\frac{n}{2\phi_{\max}}
$$

Let us refer to the minimum value of $\phi_{\max}/n$ over all possible such flows as the \emph{congestion} of $G$. The above remark can be rephrased as follows.

\begin{lem}\label{lem:congestion}
A graph of congestion at most $\rho$ has expansion at least $1/(2\rho)$. 
\end{lem}

Let us now turn our attention to the case of half-integral zonotopes. It is well-known that hypercubes have expansion exactly $1$~\cite{Harper1964,Lindsey1964,Harper1966,Hart1976}. We give an elementary proof of the lower bound via Lemma \ref{lem:congestion} and the notion of congestion that will be useful later for the case of the graphical zonotopes of cycles.

\begin{thm}\label{thm:cube}
The graph of $[0,1]^d$ has congestion at most $1/2$.
\end{thm}
\begin{proof}
For any ordered pair $(s,t)$ of vertices of $[0,1]^d$, we construct a flow from $s$ to $t$ in the graph of $[0,1]^d$ as follows. Let $I$ denote the subset of the indices $i$ such that $s_i$ differs from $t_i$. Consider the path from $s$ to $t$ in the graph of $[0,1]^d$ obtained by iteratively flipping $s_i$ to $t_i$ in increasing order of the indices $i$ contained in $I$. We route a unit of flow between $s$ and $t$ on this path.

Consider two adjacent vertices $u$ and $v$ of $[0,1]^d$. These vertices differ in exactly one of their coordinates, say the $i$th-coordinate. Denoting by $a$ the arc from $u$ to $v$, the value of $\phi(a)$ is the number of pairs $(s,t)$ of vertices of $[0,1]^d$ such that $s$ shares its last $d-i+1$ coordinates with $v$ and $t$ its first $i$ coordinates with $u$. The total number of such pairs is $2^{i-1}2^{d-i}$ and $\phi(a)$ is therefore equal to $2^{d-1}$. Recall that $[0,1]^d$ has $2^d$ vertices. Hence, by the definition of the congestion of a graph, the graph of $[0,1]^d$ has congestion at most $1/2$.
\end{proof}

As announced it is an immediate consequence of Lemma~\ref{lem:congestion} and Theorem~\ref{thm:cube} that the expansion of hypercubes is at least 1. Let us now turn our attention to the graphical zonotopes of cycles. In order to compute the congestion of their graphs, we will make use of the following characterization.

\begin{lem}\label{lem:cz}
Consider a cycle $C$ on $d$ vertices. The graph of $Z(C)$ is isomorphic to the subgraph induced in the graph of $[0,1]^d$ by all the vertices except for the origin of $\mathbb{R}^d$ and the vertex opposite to the origin.
\end{lem}
\begin{proof}
Let us identify $\mathbb{R}^{d-1}$ with the subspace of $\mathbb{R}^d$ spanned by the first $d-1$ coordinates. Consider the linear map $\psi:\mathbb{R}^d\rightarrow\mathbb{R}^{d-1}$ that fixes $\mathbb{R}^{d-1}$ and sends the vector of the canonical basis of $\mathbb{R}^d$ whose last coordinate is non-zero to the point in $\mathbb{R}^{d-1}$ whose all coordinates are equal to $-1$. Denote by $D$ the diagonal of $[0,1]^{d-1}$ that is incident to the origin of $\mathbb{R}^{d-1}$. By construction,
$$
\psi([0,1]^d)=[0,1]^{d-1}+(-D)\mbox{.}
$$

Moreover, $\psi$ bijectively sends the edges of $[0,1]^d$ that are neither incident to the origin of $\mathbb{R}^d$ nor to the vertex opposite to the origin to the edges of $[0,1]^{d-1}+(-D)$. Since, according to Remark~\ref{sec:gz.rem.1}, $Z(C)$ is equal up to an affine transformation, to $[0,1]^{d-1}+(-D)$, this completes the proof.
\end{proof}

Using Lemma \ref{lem:cz}, we are now able to bound the congestion of the graph of a graphical zonotope of a cycle with sufficiently many vertices.
  
\begin{thm}\label{thm:cycle}
Consider a cycle $C$. If $C$ has at least four vertices, then the graph of $Z(C)$ has congestion at most $6/7$.
\end{thm}
\begin{proof}
Denote by $p$ the origin of $\mathbb{R}^d$ and by $q$ the vertex of $[0,1]^d$ opposite to $p$. Further denote by $d$ the number of vertices of $C$ and by $G$ the subgraph induced in the graph of $[0,1]^d$ by all the vertices other than $p$ and $q$. According to Lemma~\ref{lem:cz}, the graph of $Z(C)$ is isomorphic to $G$ and we shall estimate the congestion of the latter graph. Consider an ordered pair of vertices $(s,t)$ of $G$ and the flow built from $s$ to $t$ in the proof of Theorem~\ref{thm:cube}. If this flow avoids $p$ and $q$, then this is a flow on the arcs obtained by orienting the edges of $G$ and we will keep it as is. If however, this flow passes through either $p$ or $q$, then we need to modify it. Let us explain how we then reroute the flow.

By symmetry, we can assume that the flow from $s$ to $t$ built in the proof of Theorem~\ref{thm:cube} passes through $p$. Denote by $e^1$ to $e^d$ the vectors of the canonical basis of $\mathbb{R}^d$. With this notation, the considered flow from $s$ to $t$ uses the arcs from $e^i$ to $p$ and from $p$ to $e^j$ for some indices $i$ and $j$. It will be useful to keep in mind that, by the way this flow was built in the proof of Theorem~\ref{thm:cube}, $j$ must be greater than $i$. We reroute this flow on the arcs from $e^i$ to $e^i+e^j$ and from $e^i+e^j$ to $e^j$. We proceed similarly for the flows through $q$.

In the original situation from Theorem~\ref{thm:cube}, the flow on all arcs was $2^{d-1}$. Since we have only considered the flows that do not start or end at $p$ or $q$, the flow on an arc is only possibly greater than $2^{d-1}$ because of the above rerouting. Let us estimate by how much the flow has increased on the arcs from $e^i$ to $e^i+e^j$ and from $e^i+e^j$ to $e^j$ where $j$ is greater than $i$. By an argument similar to the one in the proof of Theorem~\ref{thm:cube}, the number of flows that were going through the arcs from $e^i$ to $p$ and from $p$ to $e^j$ in the original construction is exactly $2^{i-1}2^{d-j}$. Since $j$ is greater than $i$, there are at most $2^{d-2}$ such flows. In particular, the flow on the arcs from $e^i$ to $e^i+e^j$ and from $e^i+e^j$ to $e^j$ has increased by at most $2^{d-2}$ and is therefore at most $2^{d-1}+2^{d-2}$. 

Rerouting the flows through $p$ has increased the flow on a set $A$ of arcs incident to a vertex with a unique non-zero coordinate. By symmetry, rerouting the flows through $q$ has increased the flow on a set $B$ of arcs incident to a vertex with a unique coordinate equal to $0$. Under the assumption that $d$ is at least $4$, $A$ and $B$ are disjoint, and the total flow on any arc is therefore at most $2^{d-1}+2^{d-2}$. Since $G$ has $2^d-2$ vertices, the congestion of $G$ is therefore at most
$$
\frac{2^{d-1}+2^{d-2}}{2^d-2}=\frac{3}{4-2^{3-d}}
$$
which, in turn is at most $6/7$ because $d$ is at least $4$.
\end{proof}
  
It follows from Lemma \ref{lem:congestion} and Theorem~\ref{thm:cycle} that the graphical zonotope of a cycle on at least four vertices has expansion at least $7/12$. The graph of the graphical zonotope of a cycle on three vertices is a cycle on six vertices.

\begin{lem}\label{lem:hex}
The congestion of a cycle on six vertices is at most $3/4$. 
\end{lem}
\begin{proof}
Let $G$ be a cycle on six vertices. For any ordered pair $(s,t)$ of vertices of $G$ whose distance in $G$ is at most $2$, send one unit of flow on the unique shortest path from $s$ and $t$. For pairs at distance $3$, send $1/2$ unit of flow on each of the two shortest paths. Every arc is involved in exactly three paths of length at most 2 and three paths of length 3, so the maximum flow is $3+3/2$. As a consequence, the congestion of $G$ is at most
$$
\frac{1}{6}\biggl(3+\frac{3}{2}\biggr)
$$
which is equal to $3/4$ and thus provides the announced bound.
\end{proof}

According to Corollary \ref{cor:hiz}, half-integral zonotopes are, up to an affine transformation, cartesian products of hypercubes and graphical zonotopes of cycles. Hence, provided we can bound the congestion of a cartesian product of graphs in terms of the congestion of these graphs, Theorems \ref{thm:cube} and \ref{thm:cycle} allow to bound the congestion of the graph of any half-integral zonotope. The following lemma is \cite[Lemma 22]{EppsteinFrishberg2023} but we provide a sketch of proof for completeness.

\begin{lem}\label{lem:cartesian}
Consider two simple graphs $G$ and $H$. The congestion of $G\mathord{\times}H$ cannot be greater than both the congestion of $G$ and the congestion of $H$. 
\end{lem}
\begin{proof}[Sketch of Proof]
Let us refer to the cartesian product of a vertex of $G$ with $H$ as a horizontal copy of $H$ within $G\mathord{\times}H$. Similarly, the cartesian product of $G$ with a vertex of $H$ will be a vertical copy of $G$. If $u$ and $v$ are vertices of $G$ and $H$, we call $u$ the vertical coordinate of $u\mathord{\times}v$ and $v$ its horizontal coordinate.

Consider flows on $G$ and $H$ that provide the congestion of these two graphs and assign them to the vertical copies of $G$ and the horizontal copies $H$ within $G\mathord{\times}H$. This already serves all the pairs $(s,t)$ of vertices that lie in the same copy of either $G$ or $H$. For the other pairs $(s,t)$, split the path into two paths through the intermediate vertex that has the same horizontal coordinate as $s$ and the same vertical coordinate as $t$. Then use the two corresponding paths in $G$ and $H$. This has the effect of multiplying the maximum flow $\phi_G$ in the copies of $G$ by the number $n_H$ of vertices of $H$ and the maximum flow $\phi_H$ in the copies of $H$ by the number $n_G$ of vertices of $G$.

As a consequence, the congestion of $G\mathord{\times}H$ is at most
$$
\max \biggl\{\frac{\phi_Gn_H}{n_Gn_H}, \frac{\phi_Hn_G}{n_Gn_H}\biggr\}=\max \biggl\{\frac{\phi_G}{n_G}, \frac{\phi_H}{n_H}\biggr\}
$$
which, by our choice of the flows in $G$ and $H$ is the maximum between the congestion of $G$ and the congestion of $H$, as desired.
\end{proof}

Note that the expansion of cartesian products of graphs is studied in \cite{ChungTetali1998,Mohar1989,Tillich2000}. By Lemmas \ref{lem:congestion} and \ref{lem:cartesian}, the expansion of $G\mathord{\times}H$ is at least
$$
\frac{1}{2\max \{\rho_G, \rho_H \}}=\min\biggl\{ \frac{1}{2\rho_G}, \frac{1}{2\rho_H} \biggr\}
$$
where $\rho_G$ and $\rho_H$ denote the congestion of $G$ and $H$. 

We can now conclude with the proof of our second main result.

\begin{proof}[Proof of Theorem~\ref{thm:hiz}]
According to Corollary~\ref{cor:hiz} the graphs of half-integral zonotopes are isomorphic to the graphs of the cartesian products of hypercubes and graphical zonotopes of cycles. From Theorems~\ref{thm:cube}, \ref{thm:cycle}, and \ref{lem:hex}, we know that the congestion of the graphs of both these families of polytopes is bounded from above by $6/7$. According to Lemma~\ref{lem:cartesian}, this also holds for their cartesian products. It then follows from Lemma \ref{lem:congestion} that the expansion of half-integral zonotopes is bounded from below by $7/12$.
\end{proof}

\noindent{\bf Acknowledgment.}
The writing of this article was initiated while the first author was a visiting professor at the computer science department (LIPN) of the Université Paris 13 in the Fall of 2023.

\bibliography{HalfIntegralExpanders}
\bibliographystyle{ijmart}

\end{document}